\documentclass{amsart}
\usepackage{amsmath}
\usepackage{amsfonts}
\usepackage{verbatim}
\usepackage{amsfonts}
\usepackage{amsthm}
\usepackage{verbatim}
\usepackage{amsmath}
\usepackage{graphicx}
\usepackage{algorithm}
\usepackage[noend]{algpseudocode}
\usepackage{caption}
\usepackage{subcaption}
\usepackage{epstopdf}

\makeatletter
\DeclareRobustCommand*\cal{\@fontswitch\relax\mathcal}
\makeatother
\newtheorem{theorem}{Theorem}

\newtheorem{TextAlgorithm}[theorem]{Algorithm}

\newtheorem{definition}[theorem]{Definition}

\newtheorem{lemma}[theorem]{Lemma}

\newtheorem{proposition}[theorem]{Proposition}
\newtheorem{remark}[theorem]{Remark}
\theoremstyle{remark}
\newtheorem{Nremark}{Numerical Remark}
\theoremstyle{remark}

\input{tcilatex}
\begin{document}

\title[Rigorous approximation of invariant measures]{An elementary approach to rigorous approximation of invariant measures}
\author{Stefano Galatolo$^1$}
\email{$^1$ galatolo@dm.unipi.it}
\address{Dipartimento di Matematica, Universita di Pisa, Via \ Buonarroti 1,Pisa}
\author{Isaia Nisoli$^2$}
\email{$^2$ nisoli@im.ufrj.br}
\address{Instituto de Matem\'{a}tica - UFRJ
 Av. Athos da Silveira Ramos 149,
Centro de Tecnologia - Bloco C
Cidade Universitária -
Ilha do Fund\~{a}o.
Caixa Postal 68530
21941-909 Rio de Janeiro - RJ - Brasil}

\begin{abstract}

We describe a framework in which is possible to develop and implement algorithms 
 for the approximation of invariant measures of dynamical systems with a given bound on the error of the approximation.

Our approach is based on a general statement on the approximation of fixed points for
operators between normed vector spaces, allowing an explicit estimation of the
error. 

We show the flexibility of our approach by applying it  to  piecewise expanding maps and to maps with indifferent fixed points. We show how the required estimations can be implemented to  compute invariant densities up to a given error in the $L^{1}$ or $L^\infty $ distance. We also show how to use this to compute an estimation  with certified error for the entropy of those systems.

We show how several related computational and numerical issues can be solved to obtain working implementations, and experimental results on some one dimensional maps. 
\end{abstract}

\subjclass{37M25}
\keywords{Approximation of invariant measure, transfer operator, fixed point approximation, Lyapunov exponent, interval arithmetics}
\maketitle

\section{Introduction}

\paragraph{\bf Overview}

Several important features of the statistical behavior of a dynamical system are
``encoded'' in invariant measures, and in particular in the so called \emph{Physical Invariant Measures}. Those measures are the ones which represent the statistical behavior of a large set of initial conditions. 
Having quantitative information on those measures can give information on the statistical behavior for the long time evolution of the
system.

The problem of the existence and properties of such invariant measures has
become a central area of research in the modern theory of Dynamical Systems.
A big part of the results are abstract and give no quantitative precise
information on the measure. This is  a significant limitation in
the applications and gives strong motivation to the search for algorithms which are
able to compute quantitative information on the physical measure.

The problem of approximating some interesting invariant measure of dynamical systems was
quite widely studied in the literature. Some algorithm are proved to
converge to the real invariant measure (up to errors in some given metrics)
in some classes of systems. Sometime asymptotical estimates on the rate of convergence are provided (see e.g. \cite{Din93, Din94}, \cite{DelJu02,Del99}  ,\cite{BM}, \cite{Mr}, \cite{F08}), but results giving an explicit (rigorous) bound on the error
are relatively few (see e.g. \cite{KMY,BB,L,PJ99,H}).

The biggest part of the methods and the results already known hence give not a rigorous bound on the error
which is made in the approximation. In this way,  the result of a single (finite) computation as the ones we can perform on everyday computers has not a precise mathematical meaning. If we implement an approach providing such
explicit bounds, the results of suitable, careful computations can be interpreted as rigorously (computer aided) proved statements on the behavior of the observed system.

In this paper we describe an approach which is able to provide algorithms to approximate interesting invariant measures with a precise bound on the error, and its practical implementation. 
The approach is quite general and is based on a quantitative statement on the sability of fixed points of operators under suitable approximations.
In our approach we focus on the estimations which are important to compute fixed points (rather than the whole spectral picture, as in \cite{L})  in a way that we can keep it as sharp as possible, trying also to use as much as possible the information that can be recovered by a suitable (and computable) finite dimensional approximation of the problem. 
The practical implementation of the method and the necessary
precise estimates are described here at various levels of generality, arriving to a complete implementation  for a class of piecewise expanding maps and a class of maps with an indifferent fixed point. We perform the estimates for the computation of the  invariant measure up to small errors in the $L^1$ norm in these cases, and also with small errors in the $L^\infty$  norm for  a class of piecewise expanding maps with higher regularity.
We also present some real computer experiment
performing the rigorous computation on interval maps, and our solution to the
nontrivial computational/numeric issues arising.

We end remarking that general, abstract results, on the computability of invariant measures are given in \cite{GalHoyRoj3}
(see also \cite{GHR}). In these papers  some negative result are also shown. Indeed, \emph{there are examples of computable\footnote{Computable, here means that the dynamics can be approximated at any accuracy
by an algorithm, see e.g. \cite{GalHoyRoj3} for precise definition.} systems 
without any computable invariant measure}. This show some subtelty in the general problem of computing the invariant measure up to a given error.

\paragraph{\bf Plan of the paper}

In section \ref{1} we show a general result regarding the approximation of
fixed points for linear operators between normed spaces. In this result
fixed points are approximated by extracting and exploiting as much information
as possible from the approximating operator. This general statement is suitable to be
applied to the Ulam approximation method and other discretizations.
In Section \ref{ulammthd} we show
how this can be done and we show an algorithm for the approximation of
invariant measures up to small errors in the $L^1$ norm for the case of piecewise expanding maps (with bounded
derivative). 

In section \ref{linfty} we show how in suitably regular systems, we can use a similar construction to compute the invariant measure, up to small errors in the $L^\infty $ norm.

In Section \ref{mann} we show how to apply the approach to a class of maps with indifferent fixed points.

In Section \ref{sec:ImplAlg} we show how to implement the algorithms in
practice. In particular we have to show a way to rapidly compute the steady
state of a large Markov chain up to a prescribed error. We also discuss several
other computational and programming issues, explaining how we have
implemented the algorithm to perform real rigorous computations on some
example of piecewise expanding maps.

 In section \ref{sec:Lyapunov}, as an application we show a rigorous estimation of the entropy (by the Lyapunov exponent) of such maps, as explained.
These estimations can be used as a benchmark for the validation of statistical methods to compute entropy from time series.

In Section \ref{sec:NumExp} we show the result of some experiments.
Here the invariant measure is
computed up to an error of less than 1\% with respect to the $L^{1}$ distance, while in section \ref{sec:NumExpHigher} we show an experiment in the $L^{\infty}$ framework.

\noindent \textbf{Acknowledgements}.\textbf{\ }
The authors wish to thank C. Liverani for his encouragement and fundamental suggestions on the
method we are going to describe.
We also would like to thank S. Luzzatto and ICTP (Trieste) for support and encouragement, B. Saussol for interesting suggestions, and W. Bahsoun for fruitful discussions on maps with indifferent fixed points.

The first author wishes  to thank ``Gruppo Nazionale per l'Analisi Matematica, la Probabilit\'a e
le loro Applicazioni'' (GNAMPA, INDAM), for financial support. 

The second author wishes to thank the CNPQ (Conselho Nacional de Pesquisas Cient\`{i}ficas e
Tecnol\`{o}gicas CNPq-Brazil).

The authors acknowledge the CINECA Award N. HP10C42W9Q, 2011 for the availability of high performance computing resources and support.

\section{Invariant measures and transfer operator}

Let $X$ be a metric space, $T:X\mapsto X$ a Borel measurable map and $\mu $
a $T$-invariant Borel probability measure. An invariant measure is a Borel
probability measure $\mu $ on $X$ such that for each measurable set $A$ it
holds $\mu (A)=\mu (T^{-1}(A))$.

A set $A$ is called $T$-invariant if $T^{-1}(A)=A\ (mod\ 0)$. The system $ (X,T,\mu )$ 
is said to be ergodic if each $T$-invariant set has total or null measure.
 In such systems the well known Birkhoff ergodic theorem says that for any $f\in L^{1}(X,\mu )$
it holds 
\begin{equation}
\underset{n\rightarrow \infty }{\lim }\frac{S_{n}^{f}(x)}{n}=\int \!{f}\,%
\mathrm{d}{\mu },  \label{Birkhoff}
\end{equation}%
for $\mu $ almost each $x$, where $S_{n}^{f}=f+f\circ T+\ldots +f\circ
T^{n-1}.$

We say that a point $x$ belongs to the basin of an invariant measure $\mu $
if (\ref{Birkhoff}) holds at $x$ for each bounded continuous $f$. In case $X$
is a manifold (possibly with boundary), a physical measure is an invariant
measure whose basin has positive Lebesgue measure (for more details and a
general survey see \cite{Y}).

\paragraph{\bf The transfer operator}

Let us consider the space $SM(X)$ of  Borel  measures with sign on $
X.$ A function $T$ between metric spaces naturally induces a function $
L:SM(X)\rightarrow SM(X)$ which is linear and is called transfer operator
(associated to $T$). Let us define $L$: if $\mu \in SM(X)$ then $L[\mu ]$
is such that
\begin{equation*}
L[\mu ](A)=\mu (T^{-1}(A)).
\end{equation*}%
Measures which are invariant for $T$ are fixed points of $L$, hence the
computation of invariant measures can be done by computing the fixed points of this
operator (restricted to a suitable Banach subspace where the interesting invariant measure is supposed to be). 
The most applied and studied strategy is to find a finite dimensional approximation for $L$ reducing the problem to the computation of the corresponding relevant eigenvectors of a finite matrix (some examples in Sections 
\ref{ulammthd}, \ref{sec:higher2} ).
 An approach to estimate the distance between a fixed point of a  discretization and a fixed point for the real operator can be based on quantitative spectral stability results given in \cite{KL}. 
The method  requires some estimation (see \cite{L}) which can not  be trivially done in
a rigorous way in a reasonable time. The approach we explain below  requires simpler assumptions and estimations, moreover a part of the required estimations will be done by the computer.

\section{A general statement on the approximation of fixed points\label{1}}

Let us consider a restriction of the transfer operator to an invariant normed subspace (often a Banach space of  measures having some regularity) $\mathcal{B\subseteq }SM(X)$  and let us denote its norm as $||\   ||_{\mathcal{B}}$. Let us still denote the restricted tranfer operator by $L$:$\mathcal{B\rightarrow B}$. Suppose it is possible to approximate $L$
in a suitable way by another operator $L_{\delta }$ for which we can
calculate fixed points and other properties.  We suppose $\delta \in {\mathbb R}$, being a parameter measuring the accuracy of the approximation (e.g. the size of a grid).

Our extent is to exploit as
much as possible the information contained in $L_{\delta }$ to approximate
fixed points of $L$.
Let us hence suppose that $f,$ $f_{\delta }\in \mathcal{B}$ are fixed
points, respectively of $L$ and $L_{\delta }$.

\begin{theorem}
\label{gen}Suppose that:

\begin{description}
\item[a)] $||L_{\delta }f-Lf||_{\mathcal{B}}<\infty $

\item[b)] $\exists \,N$ such that $||L_{\delta }^{N}(f_{\delta }-f)||_{%
\mathcal{B}}\leq \frac{1}{2}||f_{\delta }-f||_{\mathcal{B}}$

\item[c)] $L^i_{\delta }$ is continuous on $\mathcal{B}$; $\exists
 \, C_i ~s.t.\forall g\in \mathcal{B},~||L_{\delta }^{i} g||_{\mathcal{B}}\leq
C_i ||g||_{\mathcal{B}}.$
\end{description}

Then 
\begin{equation}\label{mainres}
||f_{\delta }-f||_{\mathcal{B}}\leq 2 ||L_{\delta }f-Lf||_{\mathcal{B}}  \sum_{i\in [0,N-1]}C_i .
\end{equation}
\end{theorem}

\begin{remark} We remark that the estimation for the error computed in Equation \ref{mainres} is an "a posteriori estimation", depending on $N$ and on $C_i$, quantities regarding the approximated operators, which will be estimated by the computer during the computation.
In the following we show how the above items a),b),c) are natural in the
context of approximating a fixed point of the transfer operator: 

a)  means that in some sense $L_{\delta }$ is an approximation of $L$ in the $\mathcal{B}$ norm. Moreover, the size of $ ||L_{\delta }f-Lf||_{\mathcal{B}} $ will be small if the approximation is good. This is the main ingredient to make the final error to be small in Equation \ref{mainres} .

 About b), the required $N$ will be calculated from a description of $L_{\delta }$
exploiting the fact that, under natural assumptions $L_{\delta }$
asymptotically contracts the space of zero average signed measures in 
$\mathcal{B}$.  We remark that since $L_{\delta }$ in our applications will be represented by a matrix  this can be seen as a "decay of correlation estimation with finite resolution". This estimation will be performed by the computer and will be a main ingredient for our "a posteriori" estimation of the error. 
This replaces some a priori estimations on the decay of correlation of the real system which are needed in some other approaches.
 Remark that b) also means that there is no ``projection'' of $f$
on other fixed points of $L_{\delta }$ than $f_{\delta }$. 

c) will be also estimated or computed by the way $L_{\delta }$ is defined.

We also remark that the assumptions required on the  operators $L, L_{\delta } $ are quite weak, in particular they are not required to satisfy some particular Lasota Yorke inequality. 
\end{remark}

\begin{proof}
(of Theorem \ref{gen}) The proof is a direct computation from the assumptions
\begin{eqnarray*}
||f_{\delta }-f||_{\mathcal{B}} &\leq &||L_{\delta }^{N}f_{\delta
}-L^{N}f||_{\mathcal{B}} \\
&\leq &||L_{\delta }^{N}f_{\delta }-L_{\delta }^{N}f||_{\mathcal{B}%
}+||L_{\delta }^{N}f-L^{N}f||_{\mathcal{B}} \\
&\leq &||L_{\delta }^{N}(f_{\delta }-f)||_{\mathcal{B}}+||L_{\delta
}^{N}f-L^{N}f||_{\mathcal{B}} \\
&\leq &\frac{1}{2}||f_{\delta }-f||_{\mathcal{B}}+||L_{\delta
}^{N}f-L^{N}f||_{\mathcal{B}}
\end{eqnarray*}%
(applying item b)). Hence%
\begin{equation*}
||f_{\delta }-f||_{\mathcal{B}}\leq 2||L_{\delta }^{N}f-L^{N}f||_{\mathcal{B}%
}
\end{equation*}%
but%
\begin{equation*}
L_{\delta }^{N}-L^{N}=\sum_{k=1}^{N}L_{\delta }^{N-k}(L_{\delta }-L)L^{k-1}
\end{equation*}%
hence%
\begin{eqnarray*}
(L_{\delta }^{N}-L^{N})f &=&\sum_{k=1}^{N}L_{\delta }^{N-k}(L_{\delta
}-L)L^{k-1}f \\
&=&\sum_{k=1}^{N}L_{\delta }^{N-k}(L_{\delta }-L)f
\end{eqnarray*}%
by item c), hence%
\begin{eqnarray*}
||(L^{N}-L_{\delta }^{N})f||_{\mathcal{B}} &\leq
&\sum_{k=1}^{N}C_{N-k}||(L_{\delta }-L)f||_{\mathcal{B}} \\
&\leq &||(L_{\delta }-L)f||_{\mathcal{B}}  \sum_{i\in [0,N-1]}C_i
\end{eqnarray*}%
by item a), and then%
\begin{equation*}
||f_{\delta }-f||_{\mathcal{B}}\leq 2||(L_{\delta }-L)f||_{\mathcal{B}}  \sum_{i\in [0,N-1]}C_i.
\end{equation*}
\end{proof}

\begin{remark}
We remark that by the above proof, the factor $2$ in \eqref{mainres} can be reduced as near as
wanted to $1$ by putting at item b) a factor smaller than $\frac 12 $. 
Moreover,  as $(L_{\delta }-L)f$ belongs to the space $V$ of zero total mass measures ($V=\{ \mu \ s.t. \mu(X)=0 \}$), the coefficients $C_i$ can be replaced by the operator norm of $L^i_{\delta}$ resticted to $V$.
\end{remark}

\section{Estimation with $L^{1}$ norm and Ulam method\label{ulammthd}}

We now give an example of application of the above general result to the approximation of invariant measures of dynamical systems of to small errors in the $L^1$ norm with the Ulam method, entering in more details for this case.  Let us briefly recall the basic notions.
Let us suppose now that $X$ is a manifold with boundary. In the \emph{Ulam Discretization} method the space $X$
is discretized by a partition $I_{\delta }$ (with $k$ elements) and the
system is approximated by a (finite state) Markov Chain with transition
probabilities 
\begin{equation}
P_{ij}={m(T}^{-1}{(I_{j})\cap I_{i})}/{m(I_{i})}  \label{pij}
\end{equation}%
(where $m$ is the normalized Lebesgue measure on $X$) and defining a
corresponding finite-dimensional operator $L_{\delta }$ ($L_{\delta }$
depend on the whole chosen partition but simplifying we will indicate it
with a parameter $\delta $ related to the size of the elements of the
partition) we remark that in this way, to $L_{\delta }$ corresponds a
matrix $P_{k}=(P_{ij})$ .

We remark that $L_{\delta }$ can be seen in the following way: let $%
F_{\delta }$ be the $\sigma -$algebra associated to the partition $I_{\delta
}$, then:%
\begin{equation}
L_{\delta }(f)=\mathbf{E}(L(\mathbf{E}(f|F_{\delta }))|F_{\delta }),
\label{000}
\end{equation}
(see also \cite{L}, notes 9 and 10 for some more explanations). Taking finer
and finer partitions, in certain systems  the finite dimensional model converges to
the real one and its natural invariant measure to the physical measure of
the original system, see e.g. \cite{BM, F07, F08, L}.

We now apply Theorem \ref{gen} to a more concrete case: $L^1$ estimations with Ulam discretization. Suppose that:

\begin{itemize}
\item $L_{\delta }$ is the Ulam approximation of $L$ as defined above.

\item $\mathcal{B}=L^{1}(X)$,  \footnote{To be more precise, we suppose $\mathcal{B}$ to be the space of absolutely continuous measures on $X$. We will informally identify  a measure of this kind with its density.}

\item There is an estimation for the regularity of $f$ compatible with the approximation procedure (to have the estimation needed at item $a)$ of Thm. \ref{gen}).

\end{itemize}

 As an example to explain this latter point, the norm $||f||_{\mathcal{B}^{\prime }}$  can be estimated (in some space ${\mathcal{B}^{\prime }}$ of regular measures) and an there is an estimation for the norm $||L_{\delta }-L||_{\mathcal{B}^{\prime}\rightarrow L^{1}}$  (where $||.||_{\mathcal{B}^{\prime }\rightarrow L^{1}}$ is the operator
norm, as an operator $\mathcal{B}^{\prime }\rightarrow L^{1}$)

In this way, the estimate required at item $a)$ of
Theorem \ref{gen} can be given as 
\begin{equation}\label{exrem}
||L_{\delta }f-Lf||_{L^{1}}\leq ||L_{\delta }-L||_{\mathcal{B}^{\prime
}\rightarrow L^{1}}||f||_{\mathcal{B}^{\prime }};
\end{equation}%
and we could bound the final error as 
\begin{equation*}
||f_{\delta }-f||_{L^{1}}\leq 2\sum^{N-1}_{0} C_{i}   ||L_{\delta }-L||_{\mathcal{B}^{\prime
}\rightarrow L^{1}}||f||_{\mathcal{B}^{\prime }}.
\end{equation*}
This  is possible, for example when  $L$ satisfies a Lasota Yorke inequality (see \cite{B,LY,LG} and Theorem \ref{th8}
e.g.) of the type
\begin{equation}
||L^{n}g||_{\mathcal{B}^{\prime }}\leq \lambda ^{n}||g||_{\mathcal{B}%
^{\prime }}+B||g||_{L^{1}},  \label{111}
\end{equation}
implying $||f||_{\mathcal{B}^{\prime }}\leq B$.

More details on this will be given in the next sections, where we consider certain classes of suitable maps and show how to implement the estimates needed in the approach. 

In this setting, hence, in certain classes of examples:

\begin{description}
\item[I1] a suitable estimation for the regularity of $f$ can be provided by the 
coefficients of the L-Y inequality (see also Section \ref{secitem1} below) or by other techniques, like invariant cones  (see section \ref{mann}).

\item[I2] an approximation inequality can be provided to satisfy  item  $a)$ of Theorem \ref{gen} : for example
 $||L_{\delta }-L||_{\mathcal{B}^{\prime
}\rightarrow L^{1}}$ is estimated a priori by the method of approximation
(see Section \ref{secitem2} below);

\item[I3]\label{itemI3} the integer $N$ relative to item b) in Theorem \ref{gen} can be
estimated by the matrix $P_{k}$ relative to $L_{\delta }$.
(see Section \ref{secitem3} below).

\item[I4] Since $\mathcal{B}$ $=L^{1}(X)$ and we consider the Ulam
approximation then $C_i =1$ (see Section \ref{secitem4} below).
\end{description}

Now let us discuss more precisely Item I3, which is central in this approach and whose
discussion is general. We discuss the other Items in the Subsection \ref%
{subsec:pwexp}, with precise estimations related to a particular family of
cases: the piecewise expanding maps.

\subsection{About item I3\label{secitem3}}

To compute $N$ we consider $V=\{\mu \in \mathcal{B} | \mu (X)=0\}$ and $||L_{\delta
}^{n}|_{V}||_{L^{1}\rightarrow L^{1}}$. Since $f-f_{\delta }\in V$, if we
prove 
\begin{equation*}
||L_{\delta }^{n}|_{V}||_{L^{1}\rightarrow L^{1}}<\frac{1}{2}
\end{equation*}%
we imply Item $b)$ of theorem \ref{gen}. In the Ulam approximation, $%
L_{\delta }$ is a finite rank operator, hence, once we fix a basis this is
given by a matrix. 

For the sake of simplicity we will suppose that all sets $I_{j}$ have the
same measure: $m(I_{j})=1/k$. This will simplify some notation. 

The natural basis $\{f_{1},...,f_{k}\}$ to consider is
the set of characteristic functions of the sets in the partition $I_{\delta
}.$ If $I_{\delta }=\{I_{1},...,I_{k}\}$ then $f_{i}=\frac{1}{\delta} 1_{I_{i}}$; after the
choice of this basis, the set of linear combinations of such characteristic
functions can be identified with $\mathbb{R}^{k}.$ By a small abuse of
notation we will also indicate by $V$ the set of zero average vectors in $%
\mathbb{R}^{k}$.

To determine $N$ we have to consider the matrix $P_{k}|_{V}$ associated to
the action of $L_{\delta }$ on the space of zero mean vectors with respect
to this basis and compute its operator norm $||P_{k}|_{V}||_{1}$ where$%
\footnote{$|.|_{1}$ will denote the $L^{1}$ norm on $\mathbb{R}^{n}.$}$ 
\begin{equation*}
||P_{k}|_{V}||_{1}=\underset{|v|_{1}=1}{\sup }|P(v)|_{1}.
\end{equation*}

By Equation \ref{000} the behavior of $L_{\delta }$ and its relation with $%
P_{k}$ is described by

\begin{equation*}
f\overset{\mathbf{E}|F_{\delta }\circ I^{-1}}{\rightarrow }v\overset{P_{k}}{%
\rightarrow }v^{\prime }\overset{I}{\rightarrow }f^{\prime }=L_{\delta }(f)
\end{equation*}%
where $I$:$~\mathbb{R}^{k}\rightarrow L^{1}$ is the trivial identification
of a vector in $\mathbb{R}^k$ with a piecewise constant function given by the choice of the
basis. This implies that
\begin{equation*}
||L_{\delta }||_{L^{1}\rightarrow L^{1}}\leq ||P_{k}||_{1}.
\end{equation*}%
Indeed if $f\in L^{1}$ , $||\mathbf{E}(f|F_{\delta })||_{L^{1}}\leq
||f||_{L^{1}}$ and $I$ is trivially an isometry.

Remark that if $\int f\text{dm}=0,$ then $\int E(f|F_{\delta })\text{dm}=0$
and converse, and hence 
\begin{equation*}
||L_{\delta }|_{V}||_{L^{1}}\leq ||P_{k}|_{I^{-1}(V)}||_{1}.
\end{equation*}
Since each vector is represented by a suitable step function, then $||L_{\delta }|_{V}||_{L^{1}}= ||P_{k}|_{I^{-1}(V)}||_{1}.$

The matrix corresponding to $L_{\delta }^{N}$ is $P_{k}^{N}$. 
Then
\begin{equation*}
||L_{\delta }^{N}|_{V}||_{L^{1}}= ||P_{k}^{N}|_{I^{-1}(V)}||_{1}.
\end{equation*}

Summarizing, we can have an estimation of $||L_{\delta
}^{N}|_{V}||_{L^{1}\rightarrow L^{1}}$by computing a matrix $\tilde{P}_{k,V}$
approximating $P_{k}|_{I^{-1}(V)}$ and $||\tilde{P}_{k,V}^{N}||_{1}$.

The algorithm will hence compute $||\tilde{P}_{k,V}^{j}||_{1}$ for each integer $j>0$,
computing $\tilde{P}_{k,V}^{j}$ iteratively from $\tilde{P}_{k,V}^{j-1}$, until it finds some $j$ for
which it can deduce $||P_{k}^{N}|_{I^{-1}(V)}||_{1}<\frac{1}{2}$.  This $j$ will be output as the $N$ required in
item b) of Theorem \ref{gen}.

\subsection{The algorithm}

We now present informally the general algorithm which arises from the
previous considerations for the approximation of invariant measures by our fixed point stability result. More
details on the implementation, in particular cases, are given for each step in the following
subsections.

\begin{TextAlgorithm}
\label{A1}The algorithm  works as follows:

\begin{enumerate}
\item Input the map and the partition.

\item Compute the matrix $\tilde{P}_{k}$ approximating  $L_{\delta }$ and the corresponding approximated fixed
point $\tilde{f}_{\delta }$  up to some required
approximation $\epsilon _{1}$

\item Compute $\Delta L$, an estimation for $||L_{\delta }f-Lf||_{L^{1}}$ 
up to some error $\epsilon _{2}$

\item Compute $N$ such that item $b)$ of Theorem \ref{gen} is verified as
described in item I3 above

\item If all computations end successfully, output $\tilde{f}_{\delta }$.
\end{enumerate}
\end{TextAlgorithm}

All was said before allows us to state the following

\begin{proposition}\label{11121}
$I^{-1}(\tilde{f}_{\delta })$ is an approximation of one invariant measure in $\cal B$, up to an error $\epsilon $ given by:%
\begin{equation*}
\epsilon \leq \epsilon _{1}+2N(\Delta L+\epsilon _{2})
\end{equation*}%
in the $L^{1}$ norm.
\end{proposition}

Of course is possible that some computation will not stop or that the approximation error, estimated above is not satisfying.
In this case the algorithm will be started again with a finer partition. 
With some a priori estimate on $N$ is possible to prove that in certain cases the computations will stop and the error will go to zero as $\delta \to 0$  (and even estimate the rate of convergence), see section \ref{sec:AlgorithmWorks} .

\section{The piecewise expanding case}

\label{subsec:pwexp}

We now enter in more details, showing how the previously explained
algorithm works in a concrete but nontrivial family of cases, where all the
required computations and estimations can be done.

Let
\begin{equation*}
||\mu ||=\underset{\phi \in C^{1},|\phi |_{\infty }=1}{\sup |\mu (\phi
^{\prime })|}
\end{equation*}%
this is related to bounded variation\footnote{ Recall that the variation of a function $g$ is defined as 
$$var(g)=\sup_{(x_{i})\in \textrm{Finite subdivisions of $[0,1]$}} \sum_{i\leq n}|g(x_{i})-g(x_{i+1})|. $$}:
if $||\mu||<\infty$ then $\mu$ is absolutely continuous with respect to
Lebesgue measure with $BV$ density (see \cite{L2}).

In this case $X=[0,1]$, $\mathcal{B}^{\prime }=\{\mu ,  ||\mu||<\infty  \}$. 
The dynamics we will consider is defined by a map satisfying the following requirements:

\begin{definition}
We will call a nonsingular function $T:([0,1],m)\rightarrow ([0,1],m)$ piecewise expanding  if

\begin{itemize}
\item There is a finite set of points $d_{1}=0,d_{2},...,d_{n}=1$ such that  for each $i$, $T|_{(d_{i},d_{i+1})}$ is $C^{2}$ and $\int_{[0,1]} \frac{|T''|}{(T')^2}dx<\infty$.

\item $\inf_{x\in \lbrack 0,1]}|T'(x)|>2$ on the set where it is defined.

\end{itemize}
\end{definition}

We remark that usually the definition of piecewise expanding map is weaker, in particular it is supposed $\inf_{x\in \lbrack 0,1]}|D_{x}T|>1$ for some iterate.
In concrete examples it  can be  supposed that the derivative is bigger than $2$ by considering
some iterate of $T$ ( the physical measure of the iterate is the
same).

We suppose that the map is computable, in the sense that we can compute the
probabilities $P_{ij}$ defined in \eqref{pij} up to any given accuracy. This
is the case for example, if the map has branches which are given by
analytic functions with computable coefficients.

Piecewise expanding maps have a finite set of
ergodic absolutely continuous invariant measures with bounded variation
density. If the map is topologically mixing such invariant measure is unique.

Such densities are also fixed points of the (Perron Frobenius) operator%
\footnote{Note that this operator corresponds to the above defined transfer operator,
but it acts on densities instead of measures.} $L:L^{1}[0,1]\rightarrow
L^{1}[0,1]$ defined by

\begin{equation*}
\lbrack Lf](x)=\sum_{y\in T^{-1}(x)}\frac{f(y)}{|T^{\prime }(y)|}.
\end{equation*}%

We now explain how to face all the points raised in the concrete
implementation of Algorithm \ref{A1}.

\subsubsection{About Item I1\label{secitem1}}

In this section we obtain an explicit estimation of the coefficients of the Lasota Yorke inequality for piecewise expanding maps.
We follow the approach of \cite{L2}, trying to optimize the size of the constants.

\begin{theorem}\label{th8} If $T$ is piecewise expanding as above and $\mu$ is a measure on $[0,1]$
\begin{equation*}
||L\mu ||\leq \frac{2}{_{\inf T^{\prime }}}||\mu ||+\frac{2}{\min
(d_{i}-d_{i+1})}\mu (1)+2\mu (|\frac{T^{\prime \prime }}{(T^{\prime })^{2}}%
|).
\end{equation*}
\end{theorem}

\begin{proof}
Remark that 
\begin{equation*}
L\mu (\phi ^{\prime })=\sum_{Z\in \{(d_{i},d_{i+1})|i\in (1,...,n-1)\}}L\mu
(\phi ^{\prime }\chi _{Z})
\end{equation*}
since $L\mu $ gives zero weight to the points $d_{i}$ ($L\mu $ is
absolutely continuous).

For each such $Z$ define $\phi_{Z}$ to be linear and such that $\phi
_{Z}=\phi $ on $\partial Z$, then define $\psi _{Z}=\phi -\phi_{Z},$ on $Z$, and extend it 
to $[0,1]$ by setting it to zero outside $Z$. This is a
continuous function. Moreover for each $x\in Z$%
\begin{equation*}
|\phi_{Z}^{\prime }|_{\infty }\leq \frac{2|\phi |_{\infty }}{\min
(d_{i}-d_{i+1})}
\end{equation*}%
Thus%
\begin{equation*}
|L\mu (\phi ^{\prime })|=|\sum_{Z}\mu (\psi _{Z}^{\prime }\circ T~\chi
_{T^{-1}(Z)})+\mu (\phi_{Z}^{\prime }\circ T~\chi _{T^{-1}(Z)})|
\end{equation*}%
now remark that, on $Z$ , $\psi _{Z}^{\prime }\circ T=(\frac{\psi _{Z}\circ T%
}{T^{\prime }})^{\prime }+\frac{(\psi _{Z}\circ T)T^{\prime \prime }}{%
(T^{\prime })^{2}},$then 
\begin{align*}
|L\mu (\phi ^{\prime })| \leq&|\sum_{Z}\mu ((\frac{\psi _{Z}\circ T}{%
T^{\prime }})^{\prime }~\chi _{T^{-1}(Z)})|+|\sum_{Z}\mu (\frac{(\psi_{Z} \circ
T)T^{\prime \prime }}{(T^{\prime })^{2}}~\chi _{T^{-1}(Z)})|\\&+\frac{2|\phi
|_{\infty }}{\min (d_{i}-d_{i+1})}\mu (1) \\
\leq& |\mu ((\frac{\psi _{Z}\circ T}{T^{\prime }})^{\prime
})|+2|\phi |_{\infty }\mu (|\frac{T^{\prime \prime }}{(T^{\prime })^{2}}|)+%
\frac{2|\phi |_{\infty }}{\min (d_{i}-d_{i+1})}\mu (1).
\end{align*}%
$\sum_{Z}\frac{\psi _{Z}\circ T}{T^{\prime }}$ is not $C^{1}$, but it can be
approximated as well as wanted by a $C^{1}$ function $\psi _{\epsilon }$
such that $|\psi _{\epsilon }-\sum_{Z}(\frac{\psi _{Z}\circ T}{T^{\prime }}%
)|_{\infty }$ and $\mu (|\psi _{\epsilon }-\sum_{Z}(\frac{\psi _{Z}\circ T}{%
T^{\prime }})|)$ are as small as wanted. Hence 
\begin{equation*}
|\mu ((\frac{\psi _{Z}\circ T}{T^{\prime }})^{\prime })|\leq ||\mu
||~|\frac{\psi _{Z}\circ T}{T^{\prime }}|_{\infty }\leq ||\mu ||~%
\frac{2}{_{\inf T^{\prime }}}|\phi |_{\infty }
\end{equation*}%
and%
\begin{eqnarray*}
|L\mu (\phi ^{\prime })| &\leq &||\mu ||~\frac{2}{_{\inf T^{\prime }}}|\phi
|_{\infty }+2|\phi |_{\infty }\mu (|\frac{T^{\prime \prime }}{(T^{\prime
})^{2}}|)+\frac{2|\phi |_{\infty }}{\min (d_{i}-d_{i+1})}\mu (1) \\
||L\mu || &\leq &\frac{2}{_{\inf T^{\prime }}}||\mu ||+\frac{2}{\min
(d_{i}-d_{i+1})}\mu (1)+2\mu (|\frac{T^{\prime \prime }}{(T^{\prime })^{2}}|)
\end{eqnarray*}
\end{proof}

\begin{remark}
We remark that from the above statement
it is easy to extract 
\[||L\mu ||\leq \frac{2}{_{\inf
T^{\prime }}}||\mu ||+\bigg(\frac{2}{\min (d_{i}-d_{i+1})}+2\bigg|\frac{T^{\prime
\prime }}{(T^{\prime })^{2}}\bigg|_{\infty }\bigg)|\mu |_{1}\] \ Where $|\mu |_{1}=%
\underset{|\phi |_{\infty }=1}{\sup |\mu (\phi )|}$ coincides with the $%
L^{1} $ norm for a density of $\mu $.
\end{remark}
\begin{remark}\label{rem:Bprime}
From now on, the following notation is going to be used throughout the paper
\begin{equation}
\lambda:= \frac{1}{_{\inf T^{\prime }}} \quad B':=\frac{2}{\min
(d_{i}-d_{i+1})}+2\bigg|\frac{T^{\prime
\prime }}{(T^{\prime })^{2}}\bigg|_{\infty }.
\end{equation}
These constant plays a central role in our treatment and $B'$ the biggest obstruction in getting
good estimates for the rigorous error. 
\end{remark}

We remark that once an inequality of the form
\begin{equation*}
||Lg||_{{\cal B} '}\leq 2\lambda ||g||_{{\cal B} '}+B^{\prime }||g||_{\cal B}.
\end{equation*}
 is established (with $2\lambda <1$) then, iterating, we have
\begin{equation*}
||L^{n}g||_{{\cal B} '}\leq 2^{n}\lambda ^{n}||Lg||_{{\cal B} '}+\frac{1}{1-2\lambda }B^{\prime
}||g||_{\cal B}
\end{equation*}%
obtaining the inequality in the form required at \eqref{111} and the
coefficient 
\[B=\frac{1}{1-2\lambda }B^{\prime
}\] 
which bounds $||f||$  from above, in our algorithm.

\subsubsection{About item I2\label{secitem2}}

As outlined before, on the interval $[0,1]$ we consider a partition made of intervals having length $\delta$.   As remarked in Item I2  we need an estimation on the quality of approximation by Ulam discretization.

\begin{lemma}

For piecewise expanding maps, if $L_{\delta }$ is given by the Ulam discretization as explained before and $f \in BV[0,1]$ is a fixed point of $L$ we have that
\begin{equation*}
||Lf-L_{\delta }f||_{L^{1}}\leq 2\delta ||f||
\end{equation*}
\end{lemma}

\begin{proof}
Recalling that $Lf=f$, it holds
\begin{equation*} 
||(L-L_{\delta })f||_{L^1}\leq ||\mathbf{E}(L(\mathbf{E}(f|\mathcal{F}_{\delta
})|\mathcal{F}_{\delta }))-\mathbf{E}(Lf|\mathcal{F}_{\delta }))||_{L^1}+||%
\mathbf{E}(f|\mathcal{F}_{\delta })-f||_{L^1},
\end{equation*}%
But $$\mathbf{E}(L(\mathbf{E}(f|\mathcal{F}_{\delta
})|\mathcal{F}_{\delta }))-\mathbf{E}(Lf|\mathcal{F}_{\delta }))=\mathbf{E}[L(\mathbf{E}(f|\mathcal{F}_{\delta })-f)|\mathcal{F}_{\delta }].$$
Since both $L$ and the conditional expectation are $L^1$ contractions
$$||(L-L_{\delta })f||_{L^1}\leq 2||\mathbf{E}(f|\mathcal{F}_{\delta })-f||_{L^1} .$$

It is not difficult to see that  for $f\in \mathcal{B}'$, holds 
\begin{equation*}
||\mathbf{E}(f|\mathcal{F}_{\delta })-f||_{L^1}\leq \delta \cdot ||f||.
\end{equation*}

Indeed from the definition of the norm we can see that $||f||\geq \sum_i |sup_{I_i}(f)-inf_{I_i}(f)|$, where $I_i$ are the various intervals composing $\mathcal{F}$.

By this, since $sup_{I_i}(f)\geq \mathbf{E}(f|I_i) \geq inf_{I_i}(f) $, it follows $\int _{I_i } |\mathbf{E}(f|\mathcal{F}_{\delta })-f|\leq \delta |sup_{I_i}(f)-inf_{I_i}(f)|$ and the above equation follows.

By this
$$||(L-L_{\delta })f||_{L^1}\leq  2\delta ||f||.$$

\end{proof}

\begin{remark}\label{12} This gives the  estimate which is needed at Item 3 of algorithm \ref{A1}.
We have that, when $f$ is an invariant measure, the inequality implies (see Equation \ref{exrem} )
\begin{equation*}
||Lf-L_{\delta }f||_{{L^1}}\leq  \frac{ 2}{k}B,
\end{equation*}
\end{remark}

\subsubsection{About item I4\label{secitem4}}

It is easy to see that if $L_{\delta }$ is given by the Ulam method 
\begin{equation*}
||L_{\delta }f||_{L^{1}}\leq ||f||_{L^{1}};
\end{equation*}
indeed $||Lf||_{L^{1}}\leq ||f||_{L^{1}}$ and $||E(f|F_{\delta
})||_{L^{1}}\leq ||f||_{L^{1}}$ as seen in Section \ref{secitem3} and $%
L_{\delta }$ comes from the composition of such functions.

\subsection{The algorithm works}\label{sec:AlgorithmWorks}
We show that the described algorithm can provide an estimation of the invariant measure with an error as small as wanted, if the size of the grid $\delta $ is chosen small enough.

\begin{theorem}
It is possible to compute the invariant measure of a topologically mixing
piecewise expanding map at any precision with our algorithm.
\end{theorem}

\begin{proof}

Since $L$ and $L_{\delta }$ satisfy the same Lasota Yorke inequality and $||L-L_{\delta }||_{BV\rightarrow L^{1}}\rightarrow 0$ as $\delta \rightarrow
0$, then by \cite{L} (see Proposition 3.1 and Lemma 6.1 ) the spectral gap of $L$ combined with the stability of the spectral picture, implies that  there are $
A,\lambda \in \mathbb{R},\lambda<1 $ independend of $\delta$, such that for $\delta $ small enough,  $L_{\delta }$ satisfies $||L_{\delta
}^{n}|_{V}||_{BV\rightarrow BV}\leq A\lambda ^{n}$.

Since $||\mathbf{E}(g|\mathcal{F}_{\delta
}||_{1}\geq 2
\delta ^{-1}||\mathbf{E}(g|\mathcal{F}_{\delta
}||$ this implie%
\[
||L_{\delta }^{n}||_{L^1\to L^1}\leq 2 \delta ^{-1}||L_{\delta }^{n}||_{BV\rightarrow BV}\leq 2 \delta ^{-1}A\lambda ^{n}.
\]

\noindent Hence if $n\geq \frac{\log (4A)^{-1} \delta }{\log \lambda }$, $||L_{\delta
}^{n}||_{L^1\to L^1}\leq \frac{1}{2}$. And the algorithm stop.
Moreover by Proposition \ref{11121} and Remark \ref{12} we have that up to multiplying constants, the error will be of the order $O(\delta \log \delta^{-1})$, and can be made as small as wanted as $\delta \to 0$.
\end{proof}

\begin{remark}
We remark that the above proof gives a rate of approximation of the order $O(\delta \log \delta^{-1})$ this is indeed the optimal rate of approximation for the Ulam approximation for piecewise expanding maps, as proved in \cite{BM}.  
\end{remark}

\section{Higher regularity, and $L^\infty$ estimations}\label{linfty}

In this section we explain an implementation of the general strategy to compute  invariant
measures with a rigorous error with respect to the $L^{\infty}$ norm in the case of expanding maps having $C^2$ regularity. A similar problem was faced in \cite{BB} and outlined in \cite{L} using the Keller-Liverani spectral stability result (\cite{KL}). 
We write explictly only the
arguments that differ substantially from the theory developed above and sketch the arguments
that can be deduced from the former sections.

\subsection{Higher regularity: the general framework}\label{sec:higher1} 
In this section we consider expanding maps of $S^1$; remark that also expanding Markov maps of the interval can be trated in a similar way in this framework.
\begin{definition}
Let $\tau:S^1\to S^1$ be a measurable transformation, $\tau\in C^2(S^1,S^1)$; we say $\tau$ is an expanding map of the circle if $|\tau'(x)|\geq\lambda>1$ for every $x\in S^1$.
\end{definition}

Such maps have a Lipshitz invariant density (see \cite{L2}, e.g.). Let us see how to find it with our approach.

In this section we will denote by $||.||_{\infty}$ the supremum norm on the interval and by
$||.||_{\textrm{Lip}}:=||.||_{\infty}+\textrm{Lip}(.)$ where $\textrm{Lip}(.)$ is
the Lipschitz costant of an observable.  We also denote by $C^{\textrm{Lip}}(I)$ the set of Lipschitz function over the interval. Below,  we will denote the operator norm $|| \ ||_{L^\infty\to L^\infty }$ with $|| \ ||_{\infty }$.

 Since $\tau$ satisfies a Lasota-Yorke inequality of the form
$\textrm{var}(L^n g)\leq \lambda^n \textrm{var}(g) +B||g||_1,$
Lemma 3.1 and section 3.1 of \cite{BB} give us the following. 
\begin{remark}\label{rem:M}
The $L^\infty $ operator norm of $L^n$ can be bounded by
\[||L^n||_{\infty}\leq M:=B+1.\]
\end{remark}

\begin{remark}\label{rem:LYMarkov} For Markov maps of the interval, such a L-Y inequality is proved in \cite{L2},
with coefficients 
\[\lambda\leq 1/\inf|\tau '| \quad B\leq \frac{1}{1-\lambda} \cdot
\bigg|\bigg|\frac{T''}{{T'}^2}\bigg|\bigg|_{\infty}.\]
\end{remark}

Fix now $k\geq k_0$ such that $\alpha=M\lambda^k<1$. Let $T:=\tau^k$ and let $L$ be (by abuse of
notation) the transfer operator associated to $T$; Lemma 3.3 of \cite{BB} proves that:

\begin{lemma}\label{lemma:LY-Lip}
The transfer operator $L:C^{\textrm{Lip}}(I)\to C^{\textrm{Lip}}(I)$ satisfies the following
Lasota-Yorke inequality:
\[\textrm{Lip}(Lg)\leq \alpha \textrm{Lip}(g)+B_1 ||g||_{\infty},\]
where $B_1:=\textrm{Lip}(L 1)$ (the transfer operator applied to the characteristic function of the unit interval).
For every $n\geq 1$ we have
\[||L^n g||_{\textrm{Lip}}\leq \alpha^n ||g||_{\textrm{Lip}}+M
(1+\frac{B_1}{1-\alpha})||g||_{\infty}.\] 
\end{lemma}


Suppose $\{I_i\}$ is a partition of $S^1$ such that $T|_{I_i}$ is invertible, denote by $T^{-1}_i$ the inverse. As a first remark, we give an estimate for $B_1$:
\begin{align*}
|L1(x)-L1(y)|&=|\sum_{i=1}^l \frac{1(T_i^{-1}(x))}{T'(T_i^{-1}(x))}-\sum_{i=1}^l
\frac{1(T_i^{-1}(y))}{T'(T_i^{-1}(y))}|\\
&\leq \sum_{i=1}^l |
\frac{1(T_i^{-1}(x))-1(T_i^{-1}(y))}{T'(T_i^{-1}(x))}|+\sum_{i=1}^l|\frac{1(T_i^{-1}(y))}{T'(T_i^{-1}(x))}-\frac
{ 1(T_i^{-1}(y))}{T'(T_i^{-1}(y))}|\\
&\leq l\cdot \bigg|\bigg|\frac{T''}{(T')^2}\bigg|\bigg|_{\infty} |x-y|.
\end{align*}
Therefore $B_1\leq l\cdot ||{T''}/{(T')^2}||_{\infty}$.

If $f\in C^{Lip}$ is the fixed point of $L$, from the variation L-Y inequality, we have 
\[||f||_{\infty}\leq||f||_{BV}\leq B+1.\] 

\subsection{Higher regularity: the approximation strategy}\label{sec:higher2}
We define now a discretization of the operator $L$, projecting on finite dimensional subspace of densities with higher
regularity with respect to the standard Ulam one.
This permits us to get an estimate in the $||.||_{\infty}$ norm for the approximation error.
\begin{theorem}
Let $P$ be a partition ${a_0,\ldots,a_n}$ of $S^1$ in $k$ homogeneous intervals; let $\{\phi_i\}$
be the family of functions given by
\[
  \phi_i(x) = \left\{
  \begin{array}{l l}
     k\cdot (x-a_{i-1}) & \quad x\in [a_{i-1},a_i]\\
     -k\cdot (x-a_{i+1}) & \quad x\in [a_i,a_{i+1}]\\
    0 & \quad x\in [a_{i-1},a_{i+1}]^c,\\
  \end{array} \right.
\]
where by definition $a_{-1}:=a_n$.
The finite dimensional "projection" \footnote{We warn that this is not a formally projection, in the sense that $\pi$ is not necessarily equal to $ \pi^2$. }
\[
\pi(f)(x)=\sum_j \frac{\int_{S^1} f\phi_j}{\int_{S^1} \phi_j}\cdot \phi_j(x),
\]
has the following properties
\begin{enumerate}
\item $\textrm{\textrm{Lip}}(\pi(f))\leq \textrm{\textrm{Lip}}(f);$
\item $||\pi(f)||_{\infty}\leq ||f||_{\infty}$
\item $||\pi(f)-f||_{\infty}\leq \textrm{\textrm{Lip}}(f)/k$
\end{enumerate}
\end{theorem}
\begin{proof}
Item $1$ is true since:
\begin{align*}
\textrm{\textrm{Lip}}(\pi(f))&=\frac{k}{|x_{j}-x_{i}|}\cdot\max_{i,j} |\int_{x_{j-1}}^{x_{j+1}}(f(x)-f(x+(x_j-x_i)))\phi_j(x)dx|\leq\textrm{\textrm{Lip}}(f).
\end{align*}
Item $2$ is true since
\[|\pi(f)(x)|=|\sum_i\frac{1}{\int_{S^1} \phi_i}\int_{S^1} f\phi_i dy \phi_i(x)|\leq||f||_{\infty}|\sum_i
\phi_i(x)|\leq ||f||_{\infty}. \]
Item $3$ is true since
\begin{align*}
|\pi(f)(x)-f(x)|&\leq \sum_i\frac{1}{\int_{S^1} \phi_i}\int_{S^1} \textrm{\textrm{Lip}}(f)|y-x| \phi_i(y)
dy\cdot |\phi_i(x)|\\
&\leq \textrm{\textrm{Lip}}(f)\cdot \frac{1}{k}.
\end{align*}
\end{proof}

From Lemma \ref{lemma:LY-Lip} and the properties of $\pi$ we have the following 
\begin{theorem}\label{theo:opdist}
If $f$ is a fixed point of $L$, then
\begin{align*}
||(L-\pi L \pi)f||_{\textrm{\textrm{Lip}}\rightarrow C^0} &\leq \frac{2}{k}(1+M)\textrm{\textrm{Lip}}(f).
\end{align*}
\end{theorem}
\begin{proof}
\begin{align*}
||(L-\pi L \pi) f||_{\infty}&\leq ||f-\pi f||_{\infty}+||\pi( L -L \pi)
f||_{\infty}
\end{align*}
and, from the fact that $||L||_{\infty}<M$ we have the thesis.
\end{proof}

Now we have all the ingredients to apply Theorem \ref{gen} and our algorithm, but for a
different norm. 

Computing rigorously $L \phi_i$ can be an expensive task,  we can avoid to compute it directly. Instead of computing $L_k:=\pi L\pi$, we can compute a suitable approximation $\tilde{L}_k$. This operator
is obtained by projecting on the $\{\phi_j\}$ the functions
\[\tilde{L}\phi_i(x)=\frac{1}{T'(a_i)}\phi_i\bigg(a_i+\frac{1}{T'(a_i)}(y-T(a_i))\bigg),\]
i.e., studying the operator obtained by taking on each interval $[a_{i-1},a_{i+1}]$ the
linearization $\tilde{T}$ of the map $T$.
A simple computation shows that
\begin{align*}
&||L_k\phi_i-\tilde{L}_k\phi_i||_{\infty}=\\
&\leq ||\frac{\phi_i(T^{-1}(x))}{|T'(T^{-1}(x))|}
-\frac { \phi_i( T^{-1}(x))}{|T'(x_i)| } ||_ {
\infty}+||\frac{\phi_i(T^{-1}(x))}{|T'(x_i)|}
-\frac { \phi_i(\tilde { T }^{-1}(x))}{|T'(x_i)| } ||_ { \infty}\\
&\leq\frac{4}{k^2}\cdot \bigg|\bigg|\frac{T''}{(T')^2}\bigg|\bigg|_{\infty}.
\end{align*}

\begin{remark}\label{rem:approxsmooth}
Noting that
\[||Lf-\tilde{L}_k f||_{\infty}\leq ||Lf-L_k f||_{\infty}+||L_kf-\tilde{L}_k f||_{\infty},\]
if $\tilde{v}_k$ is the eigenvector we compute using the operator $\tilde{L}_k$, we can
now express the rigorous error using Theorem \ref{gen} and the fact that the $||L_k^i||_{\infty}<M$ for every $i$ (by Remark \ref{rem:M})
\begin{align*}
||f-\tilde{v}_k||_{\infty}&\leq \frac{2}{k}\cdot N\cdot M\cdot (||L-L_k ||_{\infty}+||L_k-\tilde{L}_k||_{\infty})||f||_{\infty}\\
&\leq \frac{2}{k}\cdot N\cdot
M\bigg(2(M+1)M(1+\frac{B_1}{1-\alpha})+\frac{4}{k}\bigg|\bigg|\frac{T''}{(T')^2}\bigg|\bigg|_{\infty}\bigg)\cdot (B+1)
\end{align*}
\end{remark}

\section{Maps with indifferent fixed points}
\label{mann}

In the literature the computation of the invariant measures for such type of
maps was already discussed from  different points of view (see e.g. \cite{BBD,GalHoyRoj3,Mr}).
In particular two approaches are proposed:

\begin{itemize}
\item reduction of the problem to a piecewise expanding induced system (\cite%
{BBD})

\item direct application of a discretization method (\cite{Mr})
\end{itemize}

No explicit implementations are provided. So it is not clear what method
could be really suitable for the purpose. We implement a direct
discretization, following the general strategy descibed in our paper.

We also compute the entropy of an example of such systems. In \cite{NLS}
it is shown that statistical estimators converge slowly for these systems,
further motivating the rigorous calculation fo the entropy for such systems.

Let $0 < \alpha < 1$ and let us consider a map $T:[0,1]\rightarrow
\lbrack 0,1]$ of the following type:

\begin{enumerate}
\item $T(0)=0$ and there is a point $d\in (0,1)$ s.t. $T:[0,d)\overset{onto}{%
\rightarrow }[0,1)$, $T:[d,1)\overset{onto}{\rightarrow }[0,1]$.

\item Each branch of $T$ is increasing, convex and can be extended to a $%
C^{1}$ function; $T^{\prime }>1$ for all $x\in (0,d)\cup (d,1)$ and $T^{\prime }(0)=1$.

\item There is a constant $C\in (0,\infty )$ such that%
\begin{equation}
T(x)\geq x+Cx^{1+\alpha }.  \label{MP}
\end{equation}
\end{enumerate}

This kind of maps are now well known to have an absolutely continuous
invariant measure $f$ which is decreasing and unbounded, moreover they have
slow (polynomial) decay of correlation.

To apply our strategy we need an estimation for the regularity of $f$ (see item I1 in Section \ref{ulammthd} ).  A
useful estimation can be found in \cite{Mr} (Proposition 1.1, Theorem 1,
Equation 3, see also \cite{LSV}), indeed

\begin{proposition}
Let us consider the transfer operator $L$ associated to $T$ and the
following cone of decreasing functions 
\begin{equation*}
C_{A}=\{g\in L^{1}|g\geq
0,g~decreasing,~\int_{0}^{1}f~dm=1,\int_{0}^{x}f~dm\leq Ax^{1-\alpha }\}.
\end{equation*}

Let $A_{\ast }=((1-\alpha )Cd^{2+\alpha })^{-1}$, if $A\geq A_{\ast },$then $%
L(C_{A})\subseteq C_{A}$. Moreover the unique invariant density $f$ of $T$
is in $C_{A_{\ast }}$.
\end{proposition}

We remark (\cite{Mr}, lemma 2.1) that the if $f\in $ $C_{A}$ then $f(x)\leq
Ax^{-\alpha }$.

\subsection{Application of our strategy: items a), b),c)}\label{sec:indifferent_strategy}

Let us show the a priori estimation which is needed to start our strategy:
item a).

Let $g\in C_{A}$. Let $\pi $ be the Ulam projection with $\delta $ size
intervals: $\pi (g)=\mathbf{E}(g|\mathcal{F}_{\delta })$ and let $x_{0}=\tilde{n} \delta \in I$, 
with $\tilde{n}$ a small integer, and $g=g_{<x_{0}}+g_{>x_{0}}$ where $%
g_{<x_{0}}=g~1_{[0,x_{0})}$ and $g_{>x_{0}}=g~1_{[x_{0},1]}$.

Now

\begin{itemize}
\item $||g_{>x_{0}}-\pi g_{>x_{0}}||_{1}\leq \delta ~var(g_{>x_{0}})\leq
\delta Ax_{0}^{-\alpha }$

\item $||g_{<x_{0}}-\pi g_{<x_{0}}||_{1}\leq ||g_{<x_{0}}||_{1}\leq
Ax_{0}^{1-\alpha }$,
\end{itemize}
hence 
\begin{equation*}
||g-\pi g||_{1}\leq \delta Ax_{0}^{-\alpha }+Ax_{0}^{1-\alpha }.
\end{equation*}

\bigskip

We can take $x_{0}=\delta $ and
obtain 
\begin{equation*}
||g-\pi g||_{1}\leq 2A\delta ^{1-\alpha }.
\end{equation*}

Now let $f\in C_{A_{\ast }}$ be the invariant density. Remark that since $L$
and $\pi $ are $L^{1}$ contractions, for what is said above, $||Lf-L\pi
f||_{1}\leq ||f-\pi f||_{1}\leq 2A_{\ast }\delta ^{1-\alpha }$. Now,

\begin{eqnarray*}
||f-\pi L\pi f||_{1} &\leq &||f-\pi Lf+\pi Lf-\pi L\pi f||_{1} \\
&\leq &||f-\pi f||_{1}+||Lf-L\pi f||_{1} \\
&\leq &4A_{\ast }\delta ^{1-\alpha }.
\end{eqnarray*}

%

\bigskip

This gives the estimation needed at Item a) of Theorem \ref{gen}.

About item b) and c), since we are approximating in $L^{1}$, the discussion
is the same of the one shown in Sections \ref{secitem3} and \ref{secitem4},
thus $C_{i}\leq 1$.

\section{Implementing the algorithm}

\label{sec:ImplAlg}

In this section we explain the details in the implementation of our
algorithm  and some related numerical issue.
The main points are the computation of a rigorous approximation of the related Markov
chain \ and a fast method to approximate rigorously its steady state. We include some
implementation and numerical supplementary remarks, which can be skipped at
a first reading.

\subsection{Computing the Ulam approximation}\label{sec:UlamApprox}

To compute the matrix of the Ulam approximation, we have developed an
algorithm that computes, with a rigorous algorithm, the entries of a matrix $\tilde{P}_k$ which approximates $P_k$.
Now let us see how our algorithm computes a matrix $\tilde{P}_k  '$ which is preliminary to obtain  $\tilde{P}_k$.
Our algorithm computes each entry and the error associated to
each entry, in a way that the maximum of all these errors is is bounded by a certain quantity $\varepsilon$.
To compute the entries $P'_{ij}$ of the matrix consider each interval $I_i$ of the partition and consider two main cases: 
if $T$ is monotone on $I_i$, we can follow Algorithm \ref{alg:monotone}; 
if $T$ has a discontinuity in $I_i$ we use Algorithm \ref{alg:discontinuity}. 
In the algorithms  $\nu$ is an input costant, which is used to control the error on the coefficients.

\begin{algorithm}
\caption{Computing $\tilde{P}'_{ij}$ if $T$ is monotone on $I_i$}
\label{alg:monotone}
\begin{algorithmic}
\State Set $\tilde{P}'_{ij}=0$
\State Partition $I_i$ in $m$ intervals $I_{i,k}$ for $k=0,\ldots, m-1$
\For{$k = 0 \to m$} 
\State Compute $T(I_{i,k})$
\If {$T(I_{i,k})\subset I_j$} add $m(I_{i,k})$ to the coefficient $\tilde{P}'_{ij}$ \EndIf
\If {$T(I_{i,k})\subset (I_j)^{C}$} discard $I_{i,k}$ \EndIf
\If {$T(I_{i,k})\cap I_j\neq \emptyset$ and $T(I_{i,k})\cap (I_j)^{C}\neq \emptyset$ and $m(I_{i,k})>\nu$} divide $I_{i,k}$ in $m$ intervals and iterate the procedure \EndIf
\If {$T(I_{i,k})\cap I_j\neq \emptyset$ and $T(I_{i,k})\cap (I_j)^{C}\neq \emptyset$ and $m(I_{i,k})<\nu$} add $m(I_{i,k})$ to $\varepsilon_{ij}$, the error on the coefficient $\tilde{P}'_{ij}$ and discard $I_{i,k}$\EndIf
\EndFor
\end{algorithmic}
\end{algorithm}

\begin{algorithm}
\caption{Computing $\tilde{P}'_{ij}$ if $T$ has a discontinuity in $I_i$}
\label{alg:discontinuity}
\begin{algorithmic}
\State Set $\tilde{P}'_{ij}=0$
\State Partition $I_i$ in $m$ intervals $I_{i,k}$ for $k=1,\ldots, m$
\For{$k = 0 \to m$} 
\If {$I_{i,k}$ does not contain a discontinuity} apply Algorithm \ref{alg:monotone} to $I_{i,k}$ \EndIf
\If {$I_{i,k}$ contains the discontinuity and $m(I_{i,k})>\nu$} divide $I_{i,k}$ in $m$ intervals and iterate the procedure \EndIf
\If {$I_{i,k}$ contains the discontinuity and $m(I_{i,k})<\nu$} add $m(I_{i,k})$ to $\varepsilon_{ij}$, the error on the coefficient $\tilde{P}'_{ij}$ \EndIf
\EndFor
\end{algorithmic}
\end{algorithm}

The maximum of all the $\varepsilon_{ij}$ is really important for all our estimates: we are going to
denote it by $\varepsilon$.

We denote the matrix containing the computed coefficients
by $\tilde{P}_k'$, to distinguish it from $P_{k}$, the actual matrix of the Ulam
discretization. Please remark that $\tilde{P}'_k$ is not a stochastic matrix as we will need in the following. 
We perturb its entries to modify it and obtain a stochastic one by computing the sum of the elements for each row, subtract this number to $1$ and spread the result uniformly on each of the nonzero elements of
the row obtaining a new ``markovized'' matrix $\tilde{P}_k$.

Let $\varepsilon$ be the maximum of the errors $|\tilde{P}'_{ij}-P_{ij}|$, and let $\textrm{nnz}_i$  be the number of nonzero elements of the row. We  have that for each row $i$ the sum of its entries  differs from $1$ by at most $\textrm{nnz}_i\cdot\varepsilon$.
So, if we spread the result uniformly on each of the nonzero elements of the row we have a new matrix $\tilde{P}_k$ such that
\[ |\tilde{P}_{ij}-P_{ij}|< 2\cdot\varepsilon. \]
Let $\textrm{NNZ}=\max_i \textrm{nnz}_i$, then, the matrix
$\tilde{P}_k$ is such that 
\begin{equation*}
||P_k-\tilde{P}_k||_1<2\cdot \textrm{NNZ}\cdot {\varepsilon}.
\end{equation*}

The matrix $\tilde{P}_k$ is the matrix we are going to work with and the ``markovization'' process ensures
that the biggest eigenvalue of $\tilde{P}_k$ is $1$. 
Please note that, thanks to Theorem \ref{gen} we have a rigorous estimate of the $L^1$-distance between 
the eigenvectors of $\tilde{P}_k$ and $P_k$, as we are going to explain below.

\begin{remark}
Due to the form of \eqref{pij} we can bound the maximum number of non-zeros per row $\textrm{NNZ}\leq \sup |T'|+4$. 
\end{remark}

\subsection{Computing the $L^\infty$ approximation}\label{sec:HigherApprox}
As explained in section \ref{sec:higher2} we compute an approximation $\tilde{Q}_k$ to the matrix $Q_k$ associated to the 
operator $\tilde{L}_k$ linearizing the dynamics in correspondence of the nodes $a_0,\ldots,a_n$ of the discretization.
This permits us to express $\tilde{L}_k\phi_i$ in closed form and compute explicit 
formulas for the coefficients (finding the primitives). Using the iRRAM library (\cite{Mu}) 
we computed these coefficients so that all the digits represented in the \texttt{double} type are rigorously checked. 
Therefore, the error in the computation of the matrix $\tilde{Q}_k$ in the higher regularity case is due to the 
truncation involved in the markovization process:
\[||\tilde{Q}_k-Q_k||_{\infty}<2^{-50}=\varepsilon.\] 

\subsection{Computing rigorously the steady state vector and the error}\label{subsec:numerr}

\begin{remark}
Our algorithm and  our software work for maps which are topologically transitive.
This implies transitivity in the Markov chain approximating them.
Indeed, let $\mathring{I}_i$ and $\mathring{I}_j$ be the interior of two intervals of the partition, since the map is topologically transitive and the derivative is bounded away from zero,
there exists an $N_{ij}$ such that $T^{N_{ij}}(\mathring{I}_i)\cap \mathring{I}_j\neq\emptyset$, 
and this intersection is a union of
intervals, with nonzero measure. Therefore, if we call $\tilde{N}$ the maximum of all these $N_{ij}$
the matrix $P^{\tilde{N}}_k$ has strictly positive entries and therefore the matrix $P_k$ represents an
irreducible Markov chain. By the Perron-Frobenius theorem this implies that the steady state of the
Markov chain is unique.
\end{remark}

We want to find the steady state of
the irreducible Markov matrix $\tilde{P}_k$, to do so we use the power iteration method; given any
initial condition $b_{0}$, if we set 
\begin{equation*}
b_{l+1}=b_{l}\cdot \tilde{P}_k ,
\end{equation*}%
we have that $b_{l}$ converges to the steady state; we want to bound the numerical error of this operation from above.

In the following section we will denote by $||.||_F$ either the $1$-norm or the $\infty$-norm, 
depending in which framework are we working (the $F$ stands for finite dimensional).

We build an enclosure for the eigenvector using an idea from the proof of the Perron-Frobenius theorem
\cite[Theorem 1.1]{B}: a Markov matrix $A$ (aperiodic, irreducible) contracts the simplex $\Lambda$ of 
vectors $v$ having $1$-norm $1$.

This simplex is given by the convex combinations of the vectors $e_1,\ldots, e_k$ of the base; 
therefore, if we denote by $\textrm{Diam}_F$ the diameter in the distance induced by the norm $F$ we have
\begin{align*}
\textrm{Diam}_F(A^l \Lambda)&\leq \max_{i,j}||A^l(e_i-e_j)||_F\leq \max_{i,j}||A^l(e_1-e_j)||_F+||A^l(e_1-e_i)||_F\\
&\leq 2\max_i ||A^l(e_1-e_i)||_F.
\end{align*}

Fixed an input threshold $\varepsilon_{num}$ we iterate the vectors $\{e_1-e_j\}$, with $j=2,\ldots,n$ and look at their $F$-norm until we find an $l$ such that $\textrm{Diam}_F(A^l \Lambda)<\varepsilon_{num}$.
Therefore, for any initial condition $b_0$, iterating it $l$
times we get a vector contained in $A^l(\Lambda)$, whose numerical error is enclosed by $\varepsilon_{num}$.

\begin{Nremark}
We refer to \cite{Hi} for the following inequality about roundoff error in matrix vector multiplication, 
that we used to compute rigorously $N$ and $N_{\varepsilon}$ (as usual, $k$ is the size of the partition):
\[||\textrm{float}(Av)-Av||_F\leq \gamma_k \cdot ||A||_F ||v||_F\]
where, if $u$ is the machine precision
\[\gamma_k= \frac{ku}{1-ku}.\]
Please remark that $||A||_1=1,||v||_1\leq2$ in the Ulam case and that, since our matrix is sparse, we can substitute $k$ by $\textrm{NNZ}$ in the computation of the above constant.
\end{Nremark}

\subsection{Estimation of the rigorous error for the invariant measure}\label{sec:estimaterig}

The main issue that remains to be solved is the computation of the number of
iterations $N$ needed for the Ulam approximation $L_{\delta}$ to contract to 
$1/2$ the space of average $0$ vectors as explained in Section \ref{secitem3}.

In some way, we already assessed this question while we were computing the iterations of the
simplex; the vectors $e_1-e_j$, with $j=1,\ldots,k$ are a base for the space of average $0$
vectors, so, while computing rigorously the eigenvector, we can compute also the number of
iterations needed to contract the simplex. We have to be careful since we do not know the matrix $P_{k}$ of the
Ulam approximation $L_{\delta}$ explicitly but we know only its approximation $\tilde{P}_k$. 

Indeed (see Section \ref{secitem3}) 
\begin{equation*}
||L_{\delta }^{j}|_{V}||_1\leq ||(P_{k}^{j}-\tilde{P}_k ^{j}+\tilde{P}_k ^{j})|_{V}||_{1}\leq
||(P_{k}^{j}-\tilde{P}_k ^{j})|_{V}||_{1}+||\tilde{P}_k ^{j}|_{V}||_{1}.
\end{equation*}

We can estimate the second summand as follow
\begin{align*}
||P_{k}^j-\tilde{P}_k^j|_V||_1 &\leq \sum_{i=1}^j||P_{k}^{j-i}|_V||_1\cdot ||P_{k}-\tilde{P}_k|_V||_1\cdot
||\tilde{P}_k^{i-1}|_V||_1 \\
&\leq 2\cdot j\cdot \textrm{NNZ}\cdot\varepsilon,
\end{align*}
since $||P_{k}-\tilde{P}_k|_V||_1<2\cdot \textrm{NNZ}\cdot\varepsilon$, $||P_{k}^j|_V||_1\leq 1$
and $||\tilde{P}_k^h|_V||_1\leq 1$ for every $j,h$.
Therefore 
\begin{equation*}
||P^j_{k}|_V||_1\leq 2\cdot j\cdot \textrm{NNZ}\cdot\varepsilon+||\tilde{P}_k^j|_V ||_1.
\end{equation*}

Following the same line of thought we have, in the higher regularity case, that
\begin{equation*}
||Q^j_{k}|_V||_{\infty}\leq 2\cdot j\cdot M^2(\varepsilon+\frac{4}{k^2}\cdot \bigg|\bigg|\frac{T''}{(T')^2}\bigg|\bigg|_{\infty})+||\tilde{Q}_k^j|_V ||_{\infty}.
\end{equation*}

These two inequalities are really important for us, since they tell us that if $\varepsilon$ and $j$ are small enough we can estimate the number $N$ of iterates
needed for $P_k$ (resp. $Q_k$) to contract the space $V$ by the number of iterates needed by the matrix $\tilde{P}_k$.

\begin{Nremark} 
If $\varepsilon$ and $k$ are big, after some iterates the approximation error could hide the contraction of $\tilde{P}_k$.
Therefore, it is important to compute $\tilde{P}_k$ with a small $\varepsilon$.
\end{Nremark}

In the following we denote by $f$ the fixed point of $L$, $v_k$ the fixed point of $P_k$ (resp. $Q_k$), $v_{\varepsilon}$ the fixed point of $\tilde{P}_k$ (resp. $\tilde{Q}_k$) and by $\tilde{v}$ the numerical approximation of $v_{\varepsilon}$.
We recall now the sources of error in our computation, to make clear the last step of our algorithm: 

\begin{enumerate}
\item \label{err:dis} $||f-v_k||_F$, the discretization error, coming from the (Ulam or higher regularity) discretization of the transfer operator, whose final form was estimated in Remarks \ref{12} and \ref{rem:approxsmooth};

\item \label{err:app} $||v_k-v_{\varepsilon}||_F$, the approximation error: since we cannot compute
exactly the matrix $P_{k}$, we have to approximate it by computing a matrix $%
\tilde{P}_k$;

\item \label{err:num} $||v_{\varepsilon}-\tilde{v}||_F$, the numerical error in the computation of the
eigenvector, which was estimated in Subsection \ref{subsec:numerr},
\end{enumerate}
then
\[||f-\tilde{v}||_F\leq ||f-v_k||_F+||v_k-v_{\varepsilon}||_F+||v_{\varepsilon}-\tilde{v}||_F.\]

The last thing we need to compute to get our rigorous estimate is a bound
for the approximation error, item \ref{err:app}. We computed the number of iterates $N_{\varepsilon }$%
\footnote{%
please note that, if $\varepsilon$ is small,  $N_{\varepsilon}=N$ is expected. In the program we compute the two values indipendently,
even if in general $N_{\varepsilon}\leq N$.} needed for $\tilde{P}_k $ to contract
to $1/2$ the space of average $0$ vectors; then by using Theorem \ref{gen} we have that
\begin{equation*}
||v_{k}-v_{\varepsilon }||_{1}\leq 2N_{\varepsilon }||P_{k}-\tilde{P}_k
||_{1}||v_{k}||_{1}\leq 4N_{\varepsilon }\cdot \textrm{NNZ}\cdot \varepsilon.
\end{equation*}

In the $L^\infty$ case, the same reasoning leads to
\begin{equation*}
||v_{k}-v_{\varepsilon }||_{\infty}\leq 2N||Q_{k}-\tilde{Q}_k
||_{\infty}||v_{\varepsilon}||_{\infty}\leq 2N \cdot \varepsilon \cdot||v_{\varepsilon}||_{\infty}.
\end{equation*}
\begin{remark}
In this inequality we used $N$ instead of $N_{\epsilon}$. This is not a misprint but it is due to the fact that we have no a priori estimate of $||v_k||_{\infty}$, since we are using the piecewise linear approximation. 
To solve this issue we use Theorem \ref{gen} with $Q_k$ as $L_\delta$ 
and $\tilde{Q}_k$ as $L$ respectively.
\end{remark}

Finally, we have that, if $f$ is the invariant measure and $\tilde{v}$ is the computed vector,
using the estimate in Remark \ref{12}, the rigorous error is
\begin{equation*}
||f-\tilde{v}||_{1}\leq 2N\frac{2 B}{k}+4N_{\varepsilon }\cdot \textrm{NNZ}\cdot
\varepsilon +\varepsilon_{num}.
\end{equation*}
In the $\infty$ case, summing up all the inequalities, we get an explicit formula for the error
\begin{align*}
||f-\tilde{v}||_{\infty}\leq&\frac{2}{k}\cdot N\cdot
M\bigg(\frac{4}{k}\bigg|\bigg|\frac{T''}{(T')^2}\bigg|\bigg|_{\infty}+2(M+1)M(1+\frac{B_1}{1-\alpha})\bigg)\cdot (B+1)\\&+2N\cdot M^2(\varepsilon+\frac{4}{k^2}\bigg|\bigg|\frac{T''}{(T')^2}\bigg|\bigg|_{\infty}) (||\tilde{v}||_{\infty}+\varepsilon_{num}) +\varepsilon_{num},
\end{align*}
where $N$ is computed with respect to $||.||_{\infty}$.

\section{Rigorous computation of the Lyapunov exponent and entropy}\label{sec:Lyapunov}

The rigorous computation of the invariant density allows a rigorous estimation of the Lyapunov exponent of the system. 
These estimation can be used as a benchmark for the validation of statistical methods to compute entropy from time series.
We remark that for the experimental validation of these methods to understand how  fast they converge to the real value of the entropy an exact estimate for the value is needed. 
We give a method which can produce such estimation on interesting systems, where, an exact estimation of the entropy is not possible. This can be also  applied to systems having not a Markov structure, where the
convergence of statistical, symbolic methods may be slow (see \cite{NLS} e.g.). 
We remark that our approach gives statements on the entropy, wich can be considered as real mathematical theorems with a computer aided proof. 

The Lyapunov exponent at a point $x$, denoted by $\lambda(x)$, of a one dimensional map is defined by
\[L_{exp}(x)=\lim_{n\to+\infty}\frac{1}{n}\sum_{i=0}^{n}\log((T^i)'(x));\]
by Birkhoff ergodic theorem, we have that, relative to an ergodic invariant measure $\mu$, for $\mu$-a.e. $x$ we have that
\[L_{exp}(x)=\int_0^1 \log(|T'|)d\mu=L_{exp}.\]

Our algorithm permits us to compute the density of an invariant measure with a rigorous error bound. Suppose $\tilde{v}$ is the computed approximation for the invariant density, considered as a piecevise constant function; by Young's inequality we have that
\[\bigg|\int_0^1 \log(|T'(x)|) f(x) dx-\int_0^1 \log(|T'(x)|) \tilde{v}(x)dx\bigg|\leq
\max_{x\in[0,1]}(\log|T'(x)|)||f-\tilde{v}||_{1}.\]

Therefore, to compute the Lyapunov exponent, the only thing we have to do is to compute with a (relatively) small numerical error
the integral
\[\int_0^1 \log(|T'(x)|) \tilde{v} dx.\]

\section{Numerical experiments ($L^1$ case)} \label{sec:NumExp}

In this section we show the output of some complete experiments we made, using
the programs described above.

The code is now in an hybrid state: the routines that generate the matrix are written using
the BOOST Ublas library and can run on almost any computer, while the enclosure method for the certified computation of the
eigenvector requires a number of matrix-vector products proportional to the size of the partition: in our examples the size of the partition is $2^{20}\approx 10^6$. This forced us to implement and run our programs
in a parallel HPC enviroment, using the library PETSc and running them on the CINECA Cluster SP6.

The code for the programs, the matrices and the outputs of the cluster are found in the directory
\begin{center}
\textbf{http://poisson.phc.unipi.it/$\sim$nisoli/invmeasure/}
\end{center}

In every component where the maps are continuous, the maps are polynomials. So, we can use exact arithmetics (rationals) to
compute the matrix $\tilde{P}_k $.
Please note that the discontinuity points are irrational; this is taken care as we explained in
Section \ref{sec:UlamApprox}.

To ease  the reading of the tables of the data, here is a rapid summary of the different quantities
involved with reference to where they appear in the paper.

\begin{center}
\[
\begin{array}{llll}
\textrm{Inputs} & &\textrm{Outputs} &\\
\lambda & \textrm{L-Y inequality Remark \ref{rem:Bprime}}  &  N_{\varepsilon} & \textrm{iterates
of $\tilde{P}_k|_V$}\\
B' & \textrm{L-Y inequality Remark \ref{rem:Bprime}} & N & \textrm{iterates
of $P_{k}|_V$}\\
B  & \textrm{Bound for $||f||_{BV}$  Section \ref{ulammthd}}   & l & \textrm{iterates for the enclosure} \\
\varepsilon & \textrm{error on the matrix Section \ref{sec:UlamApprox}} & \varepsilon_{\textrm{rig}}
& \textrm{computed rigorous error} \\
\varepsilon_{num} & \textrm{numerical error Section \ref{subsec:numerr}} & L_{exp} &
\textrm{computed Lyapunov exponent}  \\

\end{array}
\]
\end{center}

\subsection{The Lanford map}

For our first numerical experiment we chose one of the maps which were
investigated in \cite{Lan}. The map $T:[0,1]\rightarrow \lbrack 0,1]$
given by 
\begin{equation*}
T:x\mapsto 2x+\frac{1}{2}x(1-x)\quad (\text{mod }1).
\end{equation*}%
What seems to be a good approximation of the invariant measure of the map is
plotted in figure $1$ of the cited article. Since this map does not comply
with the hypothesis of our article, i.e. there are some points where $%
1<|D_{x}T|\leq 2$ we study the map $T^{2}:=T\circ T$. Clearly, the invariant
measures for $T$ and $T^{2}$ coincide.

In figure \ref{fig:lanford} you can see a plot of this map and in figure \ref{fig:lanfordinvariant} you can see the plot of density of the the invariant measure we obtain through our method.

\begin{figure}[!h]
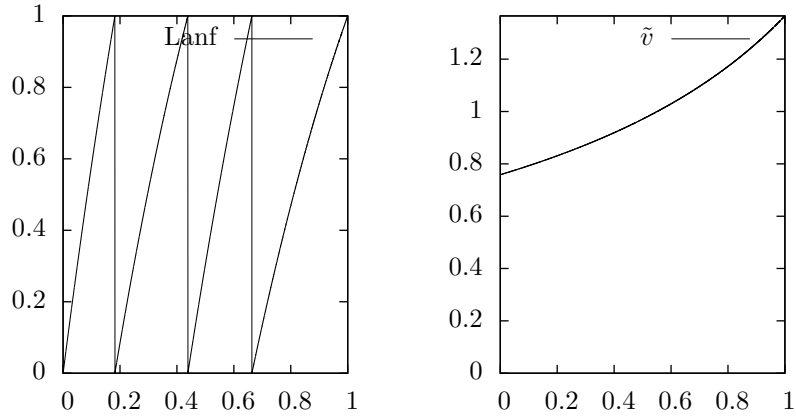

  \begin{subfigure}[b]{0.45\textwidth}
  \input{./lanfdynamic.tex}
  \caption{The second iterate of the Lanford map}
  \label{fig:lanford}

  \end{subfigure}  
  \begin{subfigure}[b]{0.45\textwidth}
  \input{./lanfdensity.tex}
 \caption{The invariant measure for the Lanford map}
    \label{fig:lanfordinvariant}
\end{subfigure}
\caption{Lanford's example.}
\end{figure}

Below, the data (input and outputs) of our algorithm.
\begin{center}
\[
\begin{array}{llll}
\textrm{Inputs} & &\textrm{Outputs} &\\
\lambda & 4/3\sqrt{17}  &  N_{\varepsilon} & 17\\
B' & \leq 7.019 & N & 18\\
B  & 19.88   & l & 25 \\
\varepsilon & \leq 3\cdot 10^{-11} & \varepsilon_{\textrm{rig}}
& 0.0016 \\
\varepsilon_{num} & \leq 0.0001 & L_{exp} & 1.315\pm 0.003  \\
\end{array}
\]
\end{center}

\subsection{A map without the Markov property}

The map $T:[0,1]\rightarrow \lbrack 0,1]$ \ given by 
\begin{equation}\label{eq:nonmarkov-2}
T(x)=\frac{17}{5}x\textrm{ mod $1$}
\end{equation}%
whose graph is plotted in figure \ref{fig:nonmarkov-2}. This map does not enjoy the Markov property:
since $(17/5)^k$ is never an integer the orbit of $1$ is dense.

The density of the invariant measure we obtain through our method is plotted in figure
\ref{fig:nonmarkov-2-invariant}.

\begin{figure}[!h]
  \begin{subfigure}[b]{0.45\textwidth}
      \input{./nonmarkovlinearfunc.tex}
      \caption{Map \eqref{eq:nonmarkov-2}}
      \label{fig:nonmarkov-2}
  \end{subfigure}  
  \begin{subfigure}[b]{0.45\textwidth}
    \input{./nonmarkovlineardens.tex}
    \caption{The invariant measure for map \eqref{eq:nonmarkov-2}}
    \label{fig:nonmarkov-2-invariant}
  \end{subfigure}
\caption{Example \eqref{eq:nonmarkov-2}.}
\end{figure}

Below, some of the data (input and outputs) of our algorithm; please note that in this case we know
the exact value for the Lyapunov exponent.
\begin{center}
\[
\begin{array}{llll}
\textrm{Inputs} & &\textrm{Outputs} &\\
\lambda & 5/17  &  N_{\varepsilon} & 13\\
B' & < 17 & N & 14\\
B  & 41.47   & l & 20 \\
\varepsilon & \leq 1.75\cdot 10^{-10} & \varepsilon_{\textrm{rig}}
& 0.0026 \\
\varepsilon_{num} & \leq 0.0001 & L_{exp} & \ln(17)-\ln(5)  \\
\end{array}
\]
\end{center}
 
\subsection{A nonlinear version}

We study the map $T:[0,1]\rightarrow \lbrack 0,1]$ \ given by 
\begin{equation}\label{eq:nonmarkov}
T(x)=\left\{ 
\begin{array}{lc}
\frac{17}{5}x & 0\leq x\leq \frac{5}{17} \\ 
\frac{34}{25}(x-\frac{5}{17})^2+3(x-\frac{5}{17}) & \frac{5}{17}<x\leq \frac{10}{17} \\ 
\frac{34}{25}(x-\frac{10}{17})^2+3(x-\frac{10}{17}) & \frac{10}{17}<x\leq \frac{15}{17} \\ 
\frac{17}{5}(x-\frac{15}{17}) & \frac{15}{17}<x\leq 1.%
\end{array}%
\right.
\end{equation}%
whose graph is plotted in figure \ref{fig:nonmarkov}.
This map is really similar to map \eqref{eq:nonmarkov-2}, but it is nonlinear in the two intervals
$[5/17,10/17]$ and $[10/17,15/17]$, where it is defined by two branches of a polynomial of degree
two.

The density of the invariant measure we obtain through our method is plotted in figure \ref{fig:nonmarkovinvariant}.
Please note that, near $0.337$ and $0.403$ there are two small ``staircase steps'' which are
visible only zooming the graph. 

\begin{figure}[!h]
  \begin{subfigure}[b]{0.45\textwidth}
      \input{./nonmarkovnonlinearfunc.tex}
      \caption{Map \eqref{eq:nonmarkov}}
      \label{fig:nonmarkov}
  \end{subfigure}  
  \begin{subfigure}[b]{0.45\textwidth}
    \input{./nonmarkovnonlineardens.tex}
    \caption{The invariant measure for map \eqref{eq:nonmarkov}}
    \label{fig:nonmarkovinvariant}
  \end{subfigure}
\caption{Example \eqref{eq:nonmarkov}.}
\end{figure}

Below, some of the data (input and outputs) of our algorithm.
\begin{center}
\[
\begin{array}{llll}
\textrm{Inputs} & &\textrm{Outputs} &\\
\lambda & 1/3  &  N_{\varepsilon} & 14\\
B' & < 18.22 & N & 15\\
B  & 54.69   & l & 21 \\
\varepsilon & \leq 2.19\cdot 10^{-11} & \varepsilon_{\textrm{rig}}
& 0.004 \\
\varepsilon_{num} & \leq 0.0001 & L_{exp} & 1.219\pm 0.004  \\
\end{array}
\]
\end{center}

\subsection{A Manneville-Pomeau map}

In this section  we compute a density with small error in the $L^1$ norm, using the estimations developed in
Sections \ref{mann}.

The numerical part is essentially the same as the one used to compute the invariant measure in the $L^1$ case, the only big difference 
resides in the fact that to compute the Ulam approximation we used an algorithm based on an interval Newton root-finding algorithm, 
instead of using the exhaustion algorithm.

The example we have studied is
\begin{equation}\label{eq:mann}
T(x)=x+x^{1+\frac18} \quad \textrm{mod $1$,}
\end{equation}
whose graph is plotted in figure \ref{fig:mann_func}, using a discretization in $1048576$ elements.

The density of the invariant measure is plotted in
figure \ref{fig:mann_dens}.

\begin{center}
\[
\begin{array}{llll}
\textrm{Inputs} & &\textrm{Outputs} &\\
\alpha & 0.125  &  N_{\varepsilon} & 49\\
A_* & \leq 4.58 & N & 50\\
\varepsilon_{num} & \leq 0.001 & l & 88 \\
d& [0.52039,0.52040] & \varepsilon_{\textrm{rig}}
& 0.006   \\
\varepsilon &\leq 2.1\cdot 10^{-15}  & L_{exp} & 0.685\pm0.005  \\
\end{array}
\]
\end{center}

\begin{figure}[!h]
  \begin{subfigure}[b]{0.45\textwidth}
      \input{mann_0125_func.tex}
      \caption{Map \eqref{eq:mann}}
      \label{fig:mann_func}
  \end{subfigure}  
  \begin{subfigure}[b]{0.45\textwidth}
    \input{./mann_0125_dens.tex}
    \caption{The invariant measure for map \eqref{eq:mann}}
    \label{fig:mann_dens}
  \end{subfigure}
\caption{Example \eqref{eq:mann}.}
\end{figure}

\section{Numerical experiments ($L^\infty$ case)} \label{sec:NumExpHigher}
In this section  we compute a density with small error in the $L^\infty$ norm, using the estimations developed in
Sections \ref{sec:higher1} and \ref{sec:higher2}, and using the methods explained in the Subsections \ref{sec:HigherApprox}, \ref{subsec:numerr}, \ref{sec:estimaterig}.

\subsection{A Markov perturbation of $4\cdot x $ mod $1$}
The example we have studied is
\begin{equation}\label{eq:higher}
T(x)=4x+0.01\cdot \sin(8\pi x) \quad \textrm{mod $1$,}
\end{equation}
whose graph is plotted in figure \ref{fig:higher}, using a discretization in $131072$ elements.

The density of the invariant measure is plotted in
figure \ref{fig:higherinvariant}.

\begin{center}
\[
\begin{array}{llll}
\textrm{Inputs} & &\textrm{Outputs} &\\
\lambda & 0.27  &  N_{\varepsilon} & 2\\
B & \leq 0.62 & N & 3\\
B_1 &<1.8 & l & 10 \\
M  & 1.62 & \varepsilon_{\textrm{rig}}
& 0.004   \\
\alpha & \leq 0.44\cdot 10^{-11} & L_{exp} & 1.386\pm0.006  \\
\varepsilon_{num} & \leq 0.00001 & &  \\
(4 ||T''/(T')^2||_{\infty})/k^2 & \leq 4\cdot 10^{-10} & & \\
\end{array}
\]
\end{center}

\begin{figure}[!h]
  \begin{subfigure}[b]{0.45\textwidth}
      \input{./higher01func.tex}
      \caption{Map \eqref{eq:higher}}
      \label{fig:higher}
  \end{subfigure}  
  \begin{subfigure}[b]{0.45\textwidth}
    \input{./higher01dens.tex}
    \caption{The invariant measure for map \eqref{eq:higher}}
    \label{fig:higherinvariant}
  \end{subfigure}
\caption{Example \eqref{eq:higher}.}
\end{figure}

\section{Conclusion and directions}

We have seen a quite general strategy to obtain  rigorous computation of
invariant measures by a  fixed point stability   statement.
We showed  theoretical and practical details of the strategy
implementation on some classes of one dimensional maps.

We remark that since the estimation for the error is a posteriori and is applied to the
discretized operator, the algorithm can  work also in systems where the
spectral gap  is not present (the indifferent fixed point ones e.g.). What is needed, is an approximation estimation
to satisfy item a) of Theorem \ref{gen}  and the discretized system to contract the zero
average vectors fast enough to make the error small.

Next natural examples where to try the strategy are multidimensional piecewise hyperbolic systems.
Typically here there will be no absolutely continuous invariant
measure, but measures having fractal support. Some (quite complicated)
functional analytic framework (see \cite{LG}, \cite{BG} e.g.)  was proved to give nice spectral properties,
but an actual implementation seems to be computationally too complex. Here
probably the use of suitable simplified anisotropic norms will be useful,
but the implementation must be able to avoid the problems arising from the
bigger dimension of the space.

\end{document}